\documentclass{article}
\usepackage{amsmath,amscd,amsfonts,amssymb,amsthm}
\usepackage{graphicx}
\usepackage{epstopdf}
\usepackage{a4wide}
\usepackage{graphicx}
\usepackage{subfigure}
\usepackage[latin1]{inputenc}
\usepackage{color}
\usepackage{caption}

\newcommand{\subj}[1]{\par\noindent{\bf Mathematics Subject Classification 2010: }#1.}
\newcommand{\keyw}[1]{\par\noindent{\bf Keywords: }#1.}

\theoremstyle{definition}
\newtheorem{definition}{Definition}
\newtheorem{theorem}{Theorem}

\newtheorem{lemma}{Lemma}
\theoremstyle{remark}

\def\a{\alpha}
\def\t{\tau}
\def\p{\psi}
\def\f{{f^{[n]}_\p}}

\def\LD{{^CD_{a+}^{\a,\p}}}
\def\RD{{^CD_{b-}^{\a,\p}}}
\def\RLD{{D_{a+}^{\a,\p}}}
\def\RRD{{D_{b-}^{\a,\p}}}
\def\LI{{I_{a+}^{\a,\p}}}
\def\RI{{I_{b-}^{\a,\p}}}

\begin{document}

\title{A Caputo fractional derivative of a function with respect to another function}

\author{Ricardo Almeida\\
{\tt ricardo.almeida@ua.pt}}

\date{Center for Research and Development in Mathematics and Applications (CIDMA)\\
Department of Mathematics, University of Aveiro, 3810--193 Aveiro, Portugal}

\maketitle


\begin{abstract}
In this paper we consider a Caputo type fractional derivative with respect to another function. Some properties, like the semigroup law, a relationship between the fractional derivative and the fractional integral, Taylor's Theorem, Fermat's Theorem, etc, are studied.
Also, a numerical method to deal with such operators, consisting in approximating the fractional derivative by a sum that depends on the first-order derivative, is presented.
Relying on examples, we show the efficiency and applicability of the method. Finally, an application of the fractional derivative, by considering a Population Growth Model, and showing that we can model more accurately the process using different kernels for the fractional operator is provided.
\end{abstract}

\subj{26A33, 34A08, 65L05, 34K60}

\keyw{Fractional calculus, semigroup law, numerical methods, population growth model}


\section{Introduction}

In ordinary calculus, higher-order derivatives $D^nf=d^nf/dx^n$, with the convention $D^0f=f$, and n-fold integrals $I^n=\int\ldots\int f \,dx$, are considered only for the particular case when $n\in\mathbb N$. The obvious extension is to try to present new definitions for derivatives and integrals with arbitrary real order $\a>0$, in which the ordinary definitions are just particular cases. This subject is known as fractional calculus, and its birth happened in 1695, when L'Hospital asked Leibniz what could be the meaning of $d^{1/2}/dx^{1/2}$. For more than two centuries, this subject was relevant only in pure mathematics, and  Euler, Fourier, Abel, Liouville, Riemann, Hadamard, among others, have studied these new fractional operators, by presenting new definitions and studying their most important properties. However, in the past decades, this subject has proven its applicability in many and different natural situations, such as viscoelasticity \cite{Fang,Sapora}, anomalous diffusion \cite{Gomez,Magin}, stochastic processes \cite{Corlay,Nguyen}, signal and image processing \cite{Yu}, fractional models and control \cite{Pooseh3,Zeng}, etc.

This is a very rich field, and for it we find several definitions for fractional integrals and for fractional derivatives \cite{Kilbas,Samko}. Just to mention a few, for fractional integrals, given an integrable function $f:[a,b]\to\mathbb R$ and a positive real number $\alpha$, we have the following
\begin{align*}
\mbox{Riemann--Liouville:}& \quad  {^{RL}I_{a+}^{\a}}f(x):=\frac{1}{\Gamma(\a)}\int_a^x(x-t)^{\a-1}f(t)\,dt,\\
\mbox{Hadamard:}& \quad  {^HI_{a+}^{\a}}f(x):=\frac{1}{\Gamma(\a)}\int_a^x\left(\ln\frac{x}{t}\right)^{\a-1}\frac{f(t)}{t}\,dt,\\
\mbox{Erd\'{e}lyi--Kober:}& \quad  {^{EK}{I}_{a+,\sigma,\eta}^{\a}}f(x):=\frac{\sigma x^{-\sigma(\a+\eta)}}{\Gamma(\a)}\int_a^xt^{\sigma\eta+\sigma-1}(x^\sigma-t^\sigma)^{\a-1}f(t)\,dt, \quad \sigma>0, \eta\in\mathbb R.\\
\end{align*}
Fractional derivatives are defined using fractional integrals, and usually are in the form $D^\a=D^n I^{n-\a}$, where $I$ can be any of the previous fractional integrals and $D^n$ is some form of differential operator. For example, if we denote $n=[\a]+1\in\mathbb N$, we have the following definitions
\begin{align*}
\mbox{Riemann--Liouville:}& \quad  {^{RL}D_{a+}^{\a}}f(x):=\left(\frac{d}{dx}\right)^n {^{RL}I_{a+}^{n-\a}}f(x),\\
\mbox{Hadamard:}& \quad  {^H{D}_{a+}^{\a}}f(x):=\left(x\frac{d}{dx}\right)^n {^H{I}_{a+}^{n-\a}}f(x),\\
\mbox{Erd\'{e}lyi--Kober:}& \quad  {^{EK}{D}_{a+,\sigma,\eta}^{\a}}f(x):=x^{-\sigma \eta}\left(\frac{1}{\sigma x^{\sigma-1}}\frac{d}{dx}\right)^n x^{\sigma (n+\eta)} {^{EK}{I}_{a+,\sigma,\eta+\a}^{n-\a}}f(x).\\
\end{align*}

To overcome the vast number of definitions, we can for instance consider general operators, from which choosing special kernels and some form of differential operator, we obtain the classical fractional integrals and derivatives \cite{Agrawal1,Klimek,Malinowska1}. For example, for the kernel $k(x,t)=x-t$ and the differential operator $d/dx$, we obtain the Riemann--Liouville fractional derivative, and for  $k(x,t)=\ln(x/t)$ and the differential $x \, d/dx$, we obtain the Hadamard fractional derivative. But, in this case, most of the fundamental laws of the derivative operator can not be obtained, due to the arbitrariness of the kernel. The problems that arise from this approach are the natural limitations to the study of the basic properties of the fractional operators. To overcome this issue, another approach is to consider the special case when the kernel is of the type $k(x,t)=\p(x)-\p(t)$ and the derivative operator is of the form $1/\p'(x) \, d/dx$. Although the kernel is still unknown, involving the function $\p$, we can already deduce some properties for the fractional operator, as we shall see in this work.

Let  $\a>0$,  $I=[a,b]$ be a finite or infinite interval, $f$  an integrable function defined on  $I$ and $\p\in C^1(I)$ an increasing function such that $\p'(x)\not=0$, for all $x\in I$.
Fractional integrals and fractional derivatives of a function $f$ with respect to another function $\p$ are defined as \cite{Kilbas,Osler,Samko}
$$\LI f(x):=\frac{1}{\Gamma(\a)}\int_a^x \p'(t)(\p(x)-\p(t))^{\a-1}f(t)\,dt,$$
and
\begin{align*} \RLD f(x):= & \left(\frac{1}{\p'(x)}\frac{d}{dx}\right)^n {I_{a+}^{n-\a,\p}}f(x)\\
=&\frac{1}{\Gamma(n-\a)}\left(\frac{1}{\p'(x)}\frac{d}{dx}\right)^n\int_a^x \p'(t)(\p(x)-\p(t))^{n-\a-1}f(t)\,dt,\end{align*}
respectively, where $n=[\a]+1$. If we consider $\p(x)=x$ or $\p(x)=\ln x$, we obtain the Riemann--Liouville and Hadamard fractional operators. If we take $\p(x)=x^\sigma$, then
$${^{EK}{I}_{a+,\sigma,\eta}^{\a}}f(x)=x^{-\sigma(\a+\eta)} \LI(x^{\sigma\eta}f(x)) \quad \mbox{and} \quad
{^{EK}{D}_{a+,\sigma,\eta}^{\a}}f(x)=x^{-\sigma\eta} \LD(x^{\sigma(\eta+\a)}f(x)).$$
Similar formulas can be presented for the right fractional integral and right fractional derivative:
$$\RI f(x):=\frac{1}{\Gamma(\a)}\int_x^b \p'(t)(\p(t)-\p(x))^{\a-1}f(t)\,dt,$$
and
\begin{align*} \RRD f(x):= & \left(-\frac{1}{\p'(x)}\frac{d}{dx}\right)^n {I_{b-}^{n-\a,\p}}f(x)\\
=&\frac{1}{\Gamma(n-\a)}\left(-\frac{1}{\p'(x)}\frac{d}{dx}\right)^n\int_x^b \p'(t)(\p(t)-\p(x))^{n-\a-1}f(t)\,dt.\end{align*}

The following semigroup property is valid for fractional integrals: if $\a,\beta>0$, then
$$\LI {I_{a+}^{\beta,\p}}f(x)={I_{a+}^{\a+\beta,\p}}f(x) \quad \mbox{and} \quad \RI {I_{b-}^{\beta,\p}}f(x)={I_{b-}^{\a+\beta,\p}}f(x).$$

We suggest \cite{Agrawal2}, where some generalizations of fractional integrals and derivatives of a function with respect to another function, involving a weight function, are presented.

The goal of this paper is to present and study some properties of a Caputo fractional derivative with respect to another function. The idea is to combine the definition of Caputo fractional derivative with the Riemann--Liouville fractional derivative with respect to another function. With this new definition, we generalize some previous works dealing with the Caputo and Caputo--Hadamard fractional derivatives. We study the main properties of this operator. To mention a few, we establish a relation between fractional integration and fractional differentiation, we study some semigroup laws, and we present an integration by parts formula, which is crucial when dealing with optimal control problems.  Since only few analytical tools are available, dealing with such operators is, in many situations, very difficult. To overcome this issue, one presents here a numerical method consisting in approximating the fractional derivative by a sum involving integer-order derivatives only. At the end, we show the importance of considering such general forms of fractional operators. Starting with the exponential growth model, we consider later the same problem described by a fractional differential equation (FDE), and we shall see that the choice of the kernel determines the accuracy of the model.

The paper is organized in the following way. In Section \ref{sec: definition}, we present the main definition, the Caputo fractional derivative of a function with respect to another function. Some fundamental properties are studied; we prove that the operators are bounded and we establish a relation between this operator and the Riemann--Liouville fractional derivative with respect to another function. In Section \ref{sec:examples}, we compute the fractional derivative of a power function and of the Mittag-Leffler function. This new fractional derivative is the inverse operation of the fractional integral operator, as it can be seen in Section \ref{sec:relation}. Then, in Section \ref{sec:semigroup}, we obtain several semigroup laws for the fractional derivative. Some results, like an integration by parts formula, a Fermat's Theorem and a Taylor's Theorem, are proven in Section \ref{sec:Miscelious}. To deal numerically with such operators, in Section \ref{sec:numerical} we prove a decomposition formula for the Caputo fractional derivative, involving the first-order derivative of the function only. This allows us later to apply standard techniques to solve the problems. In Section \ref{sec:popul}, we apply the theory to study a population growth law, by showing that, if we consider fractional derivatives with the kernel depending on another function, we may describe more efficiently the dynamics of the model.

\section{Caputo-type fractional derivative}
\label{sec: definition}

In his 1967 paper \cite{Caputo}, Caputo reformulated the definition of the Riemann--Liouville fractional derivative, by switching the order of the ordinary derivative with the fractional integral operator. By doing so, the Laplace transform of this new derivative depends on integer order initial conditions, differently from the initial conditions when we use the Riemann--Liouville fractional derivative, which involve fractional order conditions. Motivated by this concept, we present the following definition.

\begin{definition} Let  $\a>0$,  $n\in\mathbb N$, $I$ is the interval $-\infty\leq a < b \leq \infty$, $f,\p\in C^n(I)$ two functions such that $\p$ is increasing and $\p'(x)\not=0$, for all $x\in I$. The left $\p$-Caputo fractional derivative of $f$ of order $\a$ is given by
$$\LD f(x):=   {I_{a+}^{n-\a,\p}}\left(\frac{1}{\p'(x)}\frac{d}{dx}\right)^nf(x),$$
and the right $\p$-Caputo fractional derivative of $f$  by
$$\RD f(x):=   {I_{b-}^{n-\a,\p}}\left(-\frac{1}{\p'(x)}\frac{d}{dx}\right)^nf(x),$$
where
$$n=[\a]+1 \, \mbox{ for } \, \a\notin\mathbb N, \quad n=\a\, \mbox{ for } \, \a\in\mathbb N.$$
\end{definition}

To simplify notation, we will use the abbreviated symbol
$$f^{[n]}_\p f(x):= \left(\frac{1}{\p'(x)}\frac{d}{dx}\right)^nf(x).$$
From the definition, it is clear that, given $\alpha=m\in\mathbb N$,
$$\LD f(x)=f^{[m]}_\p (x) \quad \mbox{and} \quad \RD f(x)=(-1)^mf^{[m]}_\p (x), $$
and if $\a\notin\mathbb N$, then
\begin{equation}\label{def:LD}\LD f(x)=\frac{1}{\Gamma(n-\a)}\int_a^x \p'(t)(\p(x)-\p(t))^{n-\a-1}f^{[n]}_\p (t)\,dt\end{equation}
and
$$\RD f(x)=\frac{1}{\Gamma(n-\a)}\int_x^b \p'(t)(\p(t)-\p(x))^{n-\a-1}(-1)^n f^{[n]}_\p (t)\,dt.$$
In particular, when $\a\in(0,1)$, we have
$$\LD f(x)=\frac{1}{\Gamma(1-\a)}\int_a^x (\p(x)-\p(t))^{-\a}f'(t)\,dt,$$
and
$$\RD f(x)=\frac{-1}{\Gamma(1-\a)}\int_x^b(\p(t)-\p(x))^{-\a}f'(t)\,dt.$$
For some special cases of $\p$, we obtain the Caputo fractional derivative \cite{Samko}, the Caputo--Hadamard fractional derivative \cite{Baleanu1,Baleanu2} and the Caputo--Erd\'{e}lyi--Kober fractional derivative \cite{Luchko}.
From now on, we will restrict to the case $\a\notin\mathbb N$, and we study some features of this $\p$-Caputo type fractional derivative. Also, to be concise, we will prove the results only for the left fractional derivative, since the methods are similar for the right fractional derivatives, doing the necessary adjustments.

\begin{theorem}\label{alternativeFormula} Suppose that $f,\p\in C^{n+1}[a,b]$. Then, for all $\a>0$,
$$\LD f(x)=\frac{(\p(x)-\p(a))^{n-\a}}{\Gamma(n+1-\a)}f^{[n]}_\p(a)+\frac{1}{\Gamma(n+1-\a)}\int_a^x(\p(x)-\p(t))^{n-\a}\frac{d}{dt}f^{[n]}_\p (t)\,dt$$
and
$$\RD f(x)=(-1)^n\frac{(\p(b)-\p(x))^{n-\a}}{\Gamma(n+1-\a)}f^{[n]}_\p(b)-\frac{1}{\Gamma(n+1-\a)}\int_x^b(\p(t)-\p(x))^{n-\a}(-1)^n\frac{d}{dt}f^{[n]}_\p (t)\,dt.$$
\end{theorem}
\begin{proof} Integrating by parts Formula \eqref{def:LD}, with $u'(t)= \p'(t)(\p(x)-\p(t))^{n-\a-1}$ and $v(t)=f^{[n]}_\p (t)$, we obtain the desired result.
\end{proof}

From the definition, it is clear that, as $\a\to(n-1)^+$, we have
$$\lim_{\a\to(n-1)^+}\LD f(x)=\int_a^x \p'(t) \frac{1}{\p'(t)}\frac{d}{dt}f^{[n-1]}_\p (t)\,dt=f^{[n-1]}_\p (x)-f^{[n-1]}_\p (a),$$
and
$$\lim_{\a\to(n-1)^+}\RD f(x)=(-1)^{n-1}(f^{[n-1]}_\p (x)-f^{[n-1]}_\p (b)).$$
If $f,\p$ are of class $C^{n+1}$, using Theorem \ref{alternativeFormula}, we get
$$\lim_{\a\to n^-}\LD f(x)=f^{[n]}_\p (x) \quad \mbox{and} \quad \lim_{\a\to n^-}\RD f(x)=(-1)^{n}f^{[n]}_\p (x).$$
Consider the norms $\|\cdot \|_C:C[a,b]\to\mathbb R$ and $\|\cdot \|_{C^{[n]}_\p}:C^n[a,b]\to\mathbb R$ given by
$$\|f \|_C:=\max_{x\in[a,b]}|f(x)| \quad \mbox{and} \quad \| f \|_{C^{[n]}_\p}:=\sum_{k=0}^n \|f^{[k]}_\p \|_C.$$

\begin{theorem} The $\p$-Caputo fractional derivatives are bounded operators. For all $\a>0$,
$$\| \LD f\|_C \leq K \| f \|_{C^{[n]}_\p} \quad \mbox{and} \quad \| \RD f\|_C \leq K \| f \|_{C^{[n]}_\p},$$
where
$$K=\frac{(\p(b)-\p(a))^{n-\a}}{\Gamma(n+1-\a)}.$$
\end{theorem}
\begin{proof} Since $|\f(t)|\leq \| f \|_{C^{[n]}_\p}$, for all $t\in[a,b]$, we have
$$\| \LD f\|_C \leq \frac{\| f \|_{C^{[n]}_\p}}{\Gamma(n-\a)}\max_{x\in[a,b]}\int_a^x \p'(t)(\p(x)-\p(t))^{n-\a-1}dt \leq K \| f \|_{C^{[n]}_\p}.$$
\end{proof}

Following the proof we conclude that $\LD f(a)=\RD f(b)=0$ since
$$| \LD f(x)|\leq \| f \|_{C^{[n]}_\p}\frac{(\p(x)-\p(a))^{n-\a}}{\Gamma(n+1-\a)}\quad \mbox{and} \quad
| \RD f(x)| \leq \| f \|_{C^{[n]}_\p}\frac{(\p(b)-\p(x))^{n-\a}}{\Gamma(n+1-\a)}.$$

A relationship between the two types of fractional derivatives is established next.

\begin{theorem}\label{relation} If $f\in C^n[a,b]$ and  $\a>0$, then
$$\LD f(x)={D_{a+}^{\a,\p}}\left[f(x)-\sum_{k=0}^{n-1}\frac{1}{k!}(\p(x)-\p(a))^kf^{[k]}_\p(a)\right]$$
and
$$\RD f(x)={D_{b-}^{\a,\p}}\left[f(x)-\sum_{k=0}^{n-1}\frac{(-1)^k}{k!}(\p(b)-\p(x))^kf^{[k]}_\p(b)\right].$$
\end{theorem}
\begin{proof} By definition and using integration by parts, we deduce the following:
\begin{align*}
&\Gamma(n-\a)\cdot {D_{a+}^{\a,\p}}\left[f(x)-\sum_{k=0}^{n-1}\frac{f^{[k]}_\p(a)}{k!}(\p(x)-\p(a))^k\right]\\
&=\left(\frac{1}{\p'(x)}\frac{d}{dx}\right)^n\int_a^x \p'(t)(\p(x)-\p(t))^{n-\a-1}
\left[f(t)-\sum_{k=0}^{n-1}\frac{f^{[k]}_\p(a)}{k!}(\p(t)-\p(a))^k\right]  \,dt\\
&=\left(\frac{1}{\p'(x)}\frac{d}{dx}\right)^n\int_a^x \frac{\p'(t)(\p(x)-\p(t))^{n-\a}}{n-\a}
\left[f^{[1]}_\p(t)-\sum_{k=1}^{n-1}\frac{f^{[k]}_\p(a)}{(k-1)!}(\p(t)-\p(a))^{k-1}\right]  \,dt\\
&=\left(\frac{1}{\p'(x)}\frac{d}{dx}\right)^{n-1}\int_a^x \p'(t)(\p(x)-\p(t))^{n-\a-1}
\left[f^{[1]}_\p(t)-\sum_{k=1}^{n-1}\frac{f^{[k]}_\p(a)}{(k-1)!}(\p(t)-\p(a))^{k-1}\right]  \,dt.\end{align*}
Repeating this procedure, we obtain
$$\left(\frac{1}{\p'(x)}\frac{d}{dx}\right)^{n-2}\int_a^x \p'(t)(\p(x)-\p(t))^{n-\a-1}
\left[f^{[2]}_\p(t)-\sum_{k=2}^{n-1}\frac{f^{[k]}_\p(a)}{(k-2)!}(\p(t)-\p(a))^{k-2}\right]  \,dt.$$
Repeating $(n-3)$ -times, one gets at the desired formula.
\end{proof}

Thus, if for all $k=0,\ldots,n-1$, $f^{[k]}_\p(a)=0$, then $\LD f(x)={D_{a+}^{\a,\p}}f(x)$, and if $f^{[k]}_\p(b)=0$ then $\RD f(x)={D_{b-}^{\a,\p}}f(x)$.

\section{Examples}
\label{sec:examples}

\begin{lemma}\label{lemma:power} Given $\beta\in\mathbb R$, consider the functions
$$f(x)=(\p(x)-\p(a))^{\beta-1} \quad \mbox{and} \quad g(x)=(\p(b)-\p(x))^{\beta-1},$$
where $\beta>n$. Then, for $\a>0$,
$$\LD f(x)=\frac{\Gamma(\beta)}{\Gamma(\beta-\a)}(\p(x)-\p(a))^{\beta-\a-1} \quad \mbox{and} \quad
\RD g(x)=\frac{\Gamma(\beta)}{\Gamma(\beta-\a)}(\p(b)-\p(x))^{\beta-\a-1}.$$
\end{lemma}
\begin{proof} Since
$$\f(x)=\frac{\Gamma(\beta)}{\Gamma(\beta-n)}(\p(x)-\p(a))^{\beta-n-1},$$
we have
\begin{align*}\LD f(x)=&\frac{\Gamma(\beta)}{\Gamma(n-\a)\Gamma(\beta-n)}(\p(x)-\p(a))^{n-\a-1}\\
&\times\int_a^x \p'(t)\left(1-\frac{\p(t)-\p(a)}{\p(x)-\p(a)}\right)^{n-\a-1} (\p(t)-\p(a))^{\beta-n-1}dt.\end{align*}
With the change of variables $u=(\p(t)-\p(a))/(\p(x)-\p(a))$, and with the help of the Beta function
$$B(x,y)=\int_0^1t^{x-1}(1-t)^{y-1}\,dt, \quad x,y>0,$$
we obtain
$$\LD f(x)=\frac{\Gamma(\beta)}{\Gamma(n-\a)\Gamma(\beta-n)}(\p(x)-\p(a))^{\beta-\a-1}B(n-\a,\beta-n).$$
Using the following property of the Beta function
$$B(x,y)=\frac{\Gamma(x)\Gamma(y)}{\Gamma(x+y)},$$
we prove the formula.
\end{proof}

For example, for $f(x)=(\psi(x)-\psi(0))^2$, we have $ {^CD_{0+}^{\a,\p}} f(x)=2/\Gamma(3-\a)(\psi(x)-\psi(0))^{2-\a}$. Note that, when $\a=1$, we have ${^CD_{0+}^{1,\p}}=2(\psi(x)-\psi(0))$. In Figure \ref{FigExemploPotencia}, we show some graphs of $\LD f(x)$, for different values of $\a$ and different kernels $\p$.

\begin{figure}[!htbp]
\centering
\includegraphics[width=10.3cm]{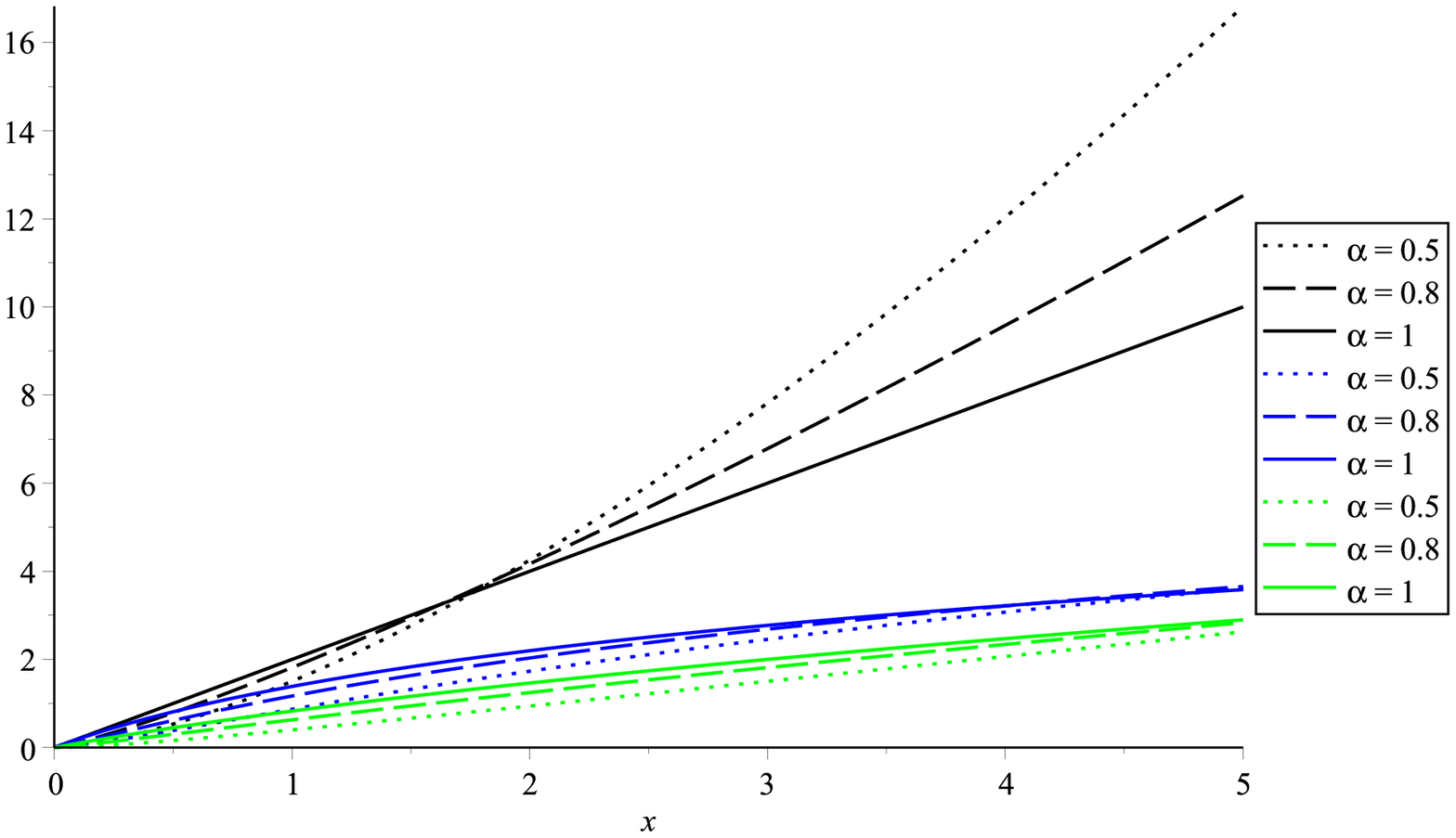}
\caption{Graph of $ {^CD_{0+}^{\a,\p}} f(x)$, for the kernels $\p(x)=x$ (black), $\p(x)=\ln(x+1)$ (blue) and $\p(x)=\sqrt{x+1}$ (green).}
\label{FigExemploPotencia}
\end{figure}

In particular, given $n \leq k \in\mathbb N$, we have
$$\LD (\p(x)-\p(a))^k=\frac{k!}{\Gamma(k+1-\a)}(\p(x)-\p(a))^{k-\a}$$
and $$
\RD (\p(b)-\p(x))^k=\frac{k!}{\Gamma(k+1-\a)}(\p(b)-\p(x))^{k-\a}.$$
On the other hand, for $n > k \in\mathbb N_0$, we have
\begin{equation}\label{aux2}\LD (\p(x)-\p(a))^k=\RD (\p(b)-\p(x))^k=0\end{equation}
since
$$D^n_\p (\p(x)-\p(a))^k= D^n_\p (\p(b)-\p(x))^k=0.$$

\begin{lemma}\label{ex:E} Given $\lambda\in\mathbb R$ and  $\a>0$, consider the functions
$$f(x)=E_\a(\lambda(\p(x)-\p(a))^\a) \quad \mbox{and} \quad g(x)=E_\a(\lambda (\p(b)-\p(x))^\a),$$
where $E_\a$ is the Mittag-Leffler function. Then,
$$\LD f(x)=\lambda f(x) \quad \mbox{and} \quad \RD g(x)=\lambda g(x).$$
\end{lemma}
\begin{proof} Using Lemma \ref{lemma:power}, we have
\begin{align*} \LD f(x) & = \sum_{k=0}^\infty\frac{\lambda^k}{\Gamma(\a k+1)}\LD ((\p(x)-\p(a))^{\a k})\\
                & = \sum_{k=1}^\infty\frac{\lambda^k}{\Gamma(\a k+1)}\frac{\Gamma(\a k+1)}{\Gamma(\a k+1-\a)} (\p(x)-\p(a))^{\a k-\a}\\
                &= \lambda  \sum_{k=1}^\infty\frac{\lambda^{k-1}}{\Gamma(\a (k-1)+1)}(\p(x)-\p(a))^{\a (k-1)}\\
                & = \lambda f(x).
\end{align*}
\end{proof}

For $f(x)=E_\a((\p(x)-\p(0))^\a)$, we have ${^CD_{0+}^{\a,\p}} f(x)=f(x)$, and for $\a=1$, we have ${^CD_{0+}^{1,\p}}=\exp(\psi(x)-\psi(0))$. In Figure \ref{FigExemploExp}, we present the graphs of ${^CD_{0+}^{\a,\p}}f(x)$, for $\a\in\{0.5,0.8,1\}$, and different kernels $\p$.

\begin{figure}[!htbp]
\centering
\includegraphics[width=10.3cm]{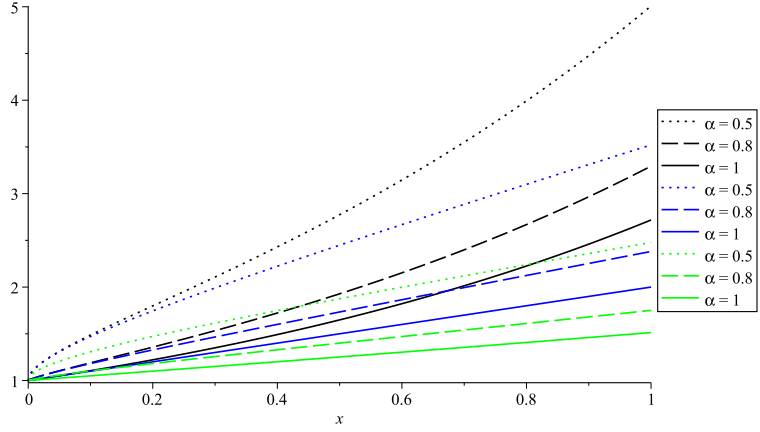}
\caption{Graph of ${^CD_{0+}^{\a,\p}} f(x)$, for the kernels $\p(x)=x$ (black), $\p(x)=\ln(x+1)$ (blue) and $\p(x)=\sqrt{x+1}$ (green).}
\label{FigExemploExp}
\end{figure}

\section{Relation between integration and derivative}
\label{sec:relation}

As we shall see in this section, the $\p$-Caputo fractional derivative is an inverse operation for the fractional integral with respect to the same function.

\begin{theorem}\label{inverseoperation1} Given a function $f\in C^n[a,b]$ and  $\a>0$, we have
$$\LI \LD f(x)=f(x)-\sum_{k=0}^{n-1}\frac{f^{[k]}_\p(a)}{k!}(\p(x)-\p(a))^k$$
and
$$\RI \RD f(x)=f(x)-\sum_{k=0}^{n-1}(-1)^k\frac{f^{[k]}_\p(b)}{k!}(\p(b)-\p(x))^k.$$
\end{theorem}
\begin{proof} Using the semigroup law and the integration by parts formula repeatedly, we get
\begin{align*}
\LI\LD f(x) & = \LI {I_{a+}^{n-\a,\p}}\f (x) = {I_{a+}^{n,\p}}\f(x)\\
            & = \frac{1}{(n-1)!}\int_a^x \p'(t)(\p(x)-\p(t))^{n-1}\f(t)\,dt\\
            & = \frac{1}{(n-1)!}\int_a^x (\p(x)-\p(t))^{n-1}\cdot \frac{d}{dt}f^{[n-1]}_\p(t)\,dt\\
            & = \frac{1}{(n-2)!}\int_a^x (\p(x)-\p(t))^{n-2}\cdot \frac{d}{dt}f^{[n-2]}_\p(t)\,dt-\frac{f^{[n-1]}_\p(a)}{(n-1)!}(\p(x)-\p(a))^{n-1}\\
            & = \frac{1}{(n-3)!}\int_a^x (\p(x)-\p(t))^{n-3}\cdot \frac{d}{dt}f^{[n-3]}_\p(t)\,dt
            -\sum_{k=n-2}^{n-1}\frac{f^{[k]}_\p(a)}{k!}(\p(x)-\p(a))^k\\
            & =\ldots =\int_a^x \frac{d}{dt}f(t)\,dt-\sum_{k=1}^{n-1}\frac{f^{[k]}_\p(a)}{k!}(\p(x)-\p(a))^k\\
            & =f(x)-\sum_{k=0}^{n-1}\frac{f^{[k]}_\p(a)}{k!}(\p(x)-\p(a))^k\\
\end{align*}
\end{proof}

In particular, given $\a\in(0,1)$, we have
$$\LI\LD f(x)=f(x)-f(a) \quad \mbox{and} \quad \RI \RD f(x)=f(x)-f(b).$$
From Theorem \ref{inverseoperation1}, we deduce a Taylor's formula:
\begin{align*}
f(x) & =\sum_{k=0}^{n-1}\frac{f^{[k]}_\p(a)}{k!}(\p(x)-\p(a))^k+\LI\LD f(x)\\
     & =\sum_{k=0}^{n-1}(-1)^k\frac{f^{[k]}_\p(b)}{k!}(\p(b)-\p(x))^k+\RI\RD f(x)\\
\end{align*}

\begin{theorem}\label{inverseoperation2} Given a function $f\in C^1[a,b]$ and  $\a>0$, we have
$$\LD \LI f(x)=f(x)\quad \mbox{and} \quad \RD \RI f(x)=f(x).$$
\end{theorem}
\begin{proof} By definition,
\begin{equation}\label{aux1}\LD \LI f(x)= \frac{1}{\Gamma(n-\a)}\int_a^x \p'(t)(\p(x)-\p(t))^{n-\a-1}F^{[n]}_\p (t)\,dt,\end{equation}
where $F(x)=\LI f(x)$. By direct computations, we get
\begin{align*}
F^{[n-1]}_\p (x) & = \frac{1}{\Gamma(\a-n+1)}\int_a^x \p'(t)(\p(x)-\p(t))^{\a-n}f(t)\,dt\\
                 & = \frac{f(a)}{\Gamma(\a-n+2)}(\p(x)-\p(a))^{\a-n+1} +\frac{1}{\Gamma(\a-n+2)}\int_a^x (\p(x)-\p(t))^{\a-n+1}f'(t)\,dt,\\
\end{align*}
and thus
$$F^{[n]}_\p (x)=\frac{f(a)}{\Gamma(\a-n+1)}(\p(x)-\p(a))^{\a-n} +\frac{1}{\Gamma(\a-n+1)}\int_a^x (\p(x)-\p(t))^{\a-n}f'(t)\,dt.$$
Replacing this last formula into Eq. \eqref{aux1}, using the change of variables $u=(\p(t)-\p(a))/(\p(x)-\p(a))$ and the Dirichlet's formula, we deduce
\begin{align*}
\LD \LI f(x)& = \frac{f(a)}{\Gamma(n-\a)\Gamma(\a-n+1)}\int_a^x \p'(t)(\p(x)-\p(t))^{n-\a-1}(\p(t)-\p(a))^{\a-n} \,dt\\
              &+ \frac{1}{\Gamma(n-\a)\Gamma(\a-n+1)}\int_a^x\int_a^t \p'(t)(\p(x)-\p(t))^{n-\a-1} (\p(t)-\p(\t))^{\a-n}f'(\t)\,d\t \,dt\\
               & = \frac{f(a)(\p(x)-\p(a))^{n-\a-1}}{\Gamma(n-\a)\Gamma(\a-n+1)}\int_a^x \p'(t)\left(1-\frac{\p(t)-\p(a)}{\p(x)-\p(a)}\right)^{n-\a-1}(\p(t)-\p(a))^{\a-n} \,dt\\
              & + \frac{1}{\Gamma(n-\a)\Gamma(\a-n+1)}\int_a^x\int_t^x \p'(\t)(\p(x)-\p(\t))^{n-\a-1} (\p(\t)-\p(t))^{\a-n}f'(t)\,d\t \,dt\\
                & = \frac{f(a)}{\Gamma(n-\a)\Gamma(\a-n+1)}\int_0^1 (1-u)^{n-\a-1}u^{\a-n} \,du\\
              & + \frac{1}{\Gamma(n-\a)\Gamma(\a-n+1)}\int_a^xf'(t)\int_t^x \p'(\t)(\p(x)-\p(\t))^{n-\a-1} (\p(\t)-\p(t))^{\a-n}\,d\t \,dt\\
                & = \frac{f(a)}{\Gamma(n-\a)\Gamma(\a-n+1)}\cdot \Gamma(n-\a)\Gamma(\a-n+1)\\
              &+\frac{1}{\Gamma(n-\a)\Gamma(\a-n+1)}\int_a^x f'(t) \, dt\cdot \Gamma(n-\a)\Gamma(\a-n+1)\\
                & = f(x).
\end{align*}
\end{proof}

\begin{theorem} Let $f,g\in C^n[a,b]$ and  $\a>0$. Then,
$$\LD f(x)=\LD g(x) \Leftrightarrow f(x)=g(x)+\sum_{k=0}^{n-1}c_k(\p(x)-\p(a))^k$$
and
$$\RD f(x)=\RD g(x) \Leftrightarrow f(x)=g(x)+\sum_{k=0}^{n-1}d_k(\p(b)-\p(x))^k,$$
where $c_k$ and $d_k$ are arbitrary constants.
\end{theorem}
\begin{proof} If  $\LD f(x)=\LD g(x)$, that is, $\LD (f(x)-g(x))=0$, applying the left integral operator to both sides of the equality, and using Theorem \ref{inverseoperation1}, we get
$$f(x)=g(x)+\sum_{k=0}^{n-1}\frac{(f-g)^{[k]}_\p(a)}{k!}(\p(x)-\p(a))^k$$
and we prove the first part of the result, with $c_k=(f-g)^{[k]}_\p(a)/k!$. To prove the reverse, assume that
$$f(x)=g(x)+\sum_{k=0}^{n-1}c_k(\p(x)-\p(a))^k.$$
Recalling Eq. \eqref{aux2}, we have
$$\LD (\p(x)-\p(a))^k=0, \quad k=0,1,\ldots,n-1.$$
With this the proof is now complete.
\end{proof}

\section{Semigroup laws}
\label{sec:semigroup}

In this section we will study several cases of composition between fractional integrals and fractional derivatives. We shall see that, in general, the semigroup law fails for the $\p$-Caputo fractional derivative, although for some specific cases it is valid.

\begin{theorem} If $f\in C^{m+n}[a,b]$ ($m\in\mathbb N$) and $\a>0$, then for all $k\in\mathbb N$ we have
$$(\LI)^k(\LD)^m f(x)=\frac{(\LD)^m f(c)}{\Gamma(k\a+1)}(\p(x)-\p(a))^{k\a}$$
and
$$(\RI)^k(\RD)^m f(x)=\frac{(\RD)^m f(d)}{\Gamma(k\a+1)}(\p(b)-\p(x))^{k\a},$$
for some $c\in(a,x)$ and  $d\in(x,b)$.
\end{theorem}
\begin{proof} Observe that, using the semigroup property for fractional integrals, we get
$$(\LI)^k:= \LI \ldots \LI = {I_{a+}^{k\a,\p}}.$$
So,
\begin{align*}
(\LI)^k(\LD)^m f(x)& ={I_{a+}^{k\a,\p}}(\LD)^m f(x)\\
          & = \frac{1}{\Gamma(k\a)}\int_a^x \p'(t)(\p(x)-\p(t))^{k\a-1}(\LD)^mf(t)\,dt\\
          & = \frac{(\LD)^mf(c)}{\Gamma(k\a)}\int_a^x \p'(t)(\p(x)-\p(t))^{k\a-1}\,dt\\
          & =\frac{(\LD)^m f(c)}{\Gamma(k\a+1)}(\p(x)-\p(a))^{k\a},
\end{align*}
for some  $c\in(a,x)$, whose existence is guaranteed by the mean value theorem for integrals.
\end{proof}

\begin{theorem}\label{aux4} Given $f\in C^{n+m}[a,b]$, with $m\in\mathbb N$, and $\a>0$, we have
$$\LD\left(\frac{1}{\p'(x)}\frac{d}{dx}\right)^mf(x)={^CD_{a+}^{\a+m,\p}}f(x) \quad \mbox{and} \quad
\RD\left(-\frac{1}{\p'(x)}\frac{d}{dx}\right)^mf(x)={^CD_{b-}^{\a+m,\p}}f(x).$$
\end{theorem}
\begin{proof} We have
\begin{align*} \LD\left(\frac{1}{\p'(x)}\frac{d}{dx}\right)^mf(x) & =
\frac{1}{\Gamma(n-\a)}\int_a^x \p'(t)(\p(x)-\p(t))^{n-\a-1}\left(\frac{1}{\p'(x)}\frac{d}{dx}\right)^mf^{[n]}_\p (t)\,dt\\
& =\frac{1}{\Gamma((n+m)-(\a+m))}\int_a^x \p'(t)(\p(x)-\p(t))^{(n+m)-(\a+m)-1}f^{[n+m]}_\p (t)\,dt\\
& = {^CD_{a+}^{\a+m,\p}}f(x).
\end{align*}
\end{proof}

We remark that, in general,
$$\left(\frac{1}{\p'(x)}\frac{d}{dx}\right)^m\LD f(x)\not={^CD_{a+}^{\a+m,\p}}f(x) \quad \mbox{and} \quad
\left(-\frac{1}{\p'(x)}\frac{d}{dx}\right)^m\RD f(x)\not={^CD_{b-}^{\a+m,\p}}f(x).$$
Just note that it is false when we consider e.g. $\p(x)=x$, that is, the (usual) Caputo fractional derivative.
From Theorem \ref{aux4}, if we define $\beta:= \a-(n-1)\in(0,1)$, we have
$$\LD f(x)= {^CD_{a+}^{\beta,\p}}{f^{[n-1]}_\p}(x) \quad \mbox{and} \quad
\RD f(x)= {^CD_{b-}^{\beta,\p}} (-1)^{n-1}{f^{[n-1]}_\p}(x).$$
Thus, to compute the fractional derivative of any order $\a>0$, it is enough to know the derivative of order $\beta=\a-(n-1)\in(0,1)$.

\begin{theorem}\label{aux6} Let $m\in\mathbb N$ be an integer and $f\in C^{n+m}[a,b]$ a function. Then,
$$\left(\frac{1}{\p'(x)}\frac{d}{dx}\right)^m \LD f(x)= {^CD_{a+}^{\a+m,\p}} f(x)
+\sum_{k=0}^{m-1}\frac{(\p(x)-\p(a))^{k+n-\a-m}}{\Gamma(k+n-\a-m+1)}{f^{[k+n]}_\p}(a)$$
and
$$\left(-\frac{1}{\p'(x)}\frac{d}{dx}\right)^m \RD f(x)= {^CD_{b-}^{\a+m,\p}} f(x)
+\sum_{k=0}^{m-1}(-1)^{k+n}\frac{(\p(b)-\p(x))^{k+n-\a-m}}{\Gamma(k+n-\a-m+1)}{f^{[k+n]}_\p}(b).$$
\end{theorem}
\begin{proof} First, observe that
$${^CD_{a+}^{\a+m,\p}} f(x)=\frac{1}{\Gamma(n-\a)}\int_a^x \p'(t)(\p(x)-\p(t))^{n-\a-1}f^{[n+m]}_\p (t)\,dt.$$
On the other hand,  integrating by parts, we get
\begin{align*}
\left(\frac{1}{\p'(x)}\frac{d}{dx}\right)^m \LD f(x) & =\left(\frac{1}{\p'(x)}\frac{d}{dx}\right)^m\left[  \frac{(\p(x)-\p(a))^{n-\a}}{\Gamma(n-\a+1)}f^{[n]}_\p (a)\right.\\
   & \quad \left.+\frac{1}{\Gamma(n-\a+1)}\int_a^x(\p(x)-\p(t))^{n-\a}\frac{d}{dt}f^{[n]}_\p (t)\,dt\right]\\
& =\left(\frac{1}{\p'(x)}\frac{d}{dx}\right)^{m-1}\left[  \frac{(\p(x)-\p(a))^{n-\a-1}}{\Gamma(n-\a)}f^{[n]}_\p (a)\right.\\
   & \quad \left.
+\frac{1}{\Gamma(n-\a)}\int_a^x\p'(t)(\p(x)-\p(t))^{n-\a-1}f^{[n+1]}_\p (t)\,dt\right].\\
\end{align*}
Integrating once more time by parts, we get
\begin{align*}
\left(\frac{1}{\p'(x)}\frac{d}{dx}\right)^m \LD f(x)& =\left(\frac{1}{\p'(x)}\frac{d}{dx}\right)^{m-2}\left[  \sum_{k=0}^{1}\frac{(\p(x)-\p(a))^{k+n-\a-2}}{\Gamma(k+n-\a-1)}{f^{[k+n]}_\p}f(a)\right.\\
   & \quad \left.
+\frac{1}{\Gamma(n-\a)}\int_a^x\p'(t)(\p(x)-\p(t))^{n-\a-1}f^{[n+2]}_\p (t)\,dt\right].\\
\end{align*}
Repeating this procedure, we reach the desired result.
\end{proof}

As a consequence of Theorem \ref{aux6}, we have that if ${f^{[k]}_\p}(a)=0$, for all $k=n,n+1,\ldots,n+m-1$, then
$$\left(\frac{1}{\p'(x)}\frac{d}{dx}\right)^m \LD f(x)= {^CD_{a+}^{\a+m,\p}} f(x),$$
and  if  ${f^{[k]}_\p}(b)=0$, for all $k=n,n+1,\ldots,n+m-1$, then
$$\left(-\frac{1}{\p'(x)}\frac{d}{dx}\right)^m \RD f(x)= {^CD_{b-}^{\a+m,\p}} f(x).$$

\begin{theorem}\label{aux5} Let $\a,\beta>0$ be such that there exists some $k\in\mathbb N$ with $\beta,\a+\beta\in[k-1,k]$. Then, for $f\in C^k[a,b]$, the following holds:
$$\LD {^CD_{a+}^{\beta,\p}}f(x)={^CD_{a+}^{\a+\beta,\p}}f(x) \quad \mbox{and} \quad
\RD {^CD_{b-}^{\beta,\p}}f(x)=(-1)^{[\a+\beta]}{^CD_{b-}^{\a+\beta,\p}}f(x).$$
\end{theorem}
\begin{proof} First, suppose that $\a+\beta=k$. In this case, since $\beta\in[k-1,k)$, we have $[\beta]=k-1=\a+\beta-1$. By Theorem \ref{inverseoperation2}, we conclude that
$$\LD {^CD_{a+}^{\beta,\p}}f(x)=\LD {I_{a+}^{\a+\beta-\beta,\p}}{f^{[\a+\beta]}_\p}={f^{[\a+\beta]}_\p}= {^CD_{a+}^{\a+\beta,\p}}f(x).$$
Suppose now that $\a+\beta<k$. In that case, $\a\in(0,1)$ and $[\beta]=[\a+\beta]=k-1$. Since ${^CD_{a+}^{\beta,\p}} f(a)=0$, by Theorem \ref{relation}, we have
$$\LD {^CD_{a+}^{\beta,\p}}f(x)={D_{a+}^{\a,\p}}{^CD_{a+}^{\beta,\p}}f(x).$$
Thus, we get finally that
\begin{align*} \LD {^CD_{a+}^{\beta,\p}}f(x) & = {D_{a+}^{\a,\p}}{^CD_{a+}^{\beta,\p}}f(x)\\
                 & = \left(\frac{1}{\p'(x)}\frac{d}{dx}\right) {I_{a+}^{1-\a,\p}} {I_{a+}^{[\beta]+1-\beta,\p}} {f^{[\beta]+1}_\p} (x)\\
                 & = \left(\frac{1}{\p'(x)}\frac{d}{dx}\right) {I_{a+}^{1,\p}} {I_{a+}^{[\a+\beta]+1-(\a+\beta),\p}} {f^{[\a+\beta]+1}_\p} (x)\\
                 & = {^CD_{a+}^{\a+\beta,\p}}f(x).
\end{align*}
\end{proof}

The assumption of Theorem \ref{aux5} on the existence of the integer $k$ is essencial. For example, it fails when we consider $\p(x)=x$ (\cite[Chapter 3]{Diethelm}).

\begin{theorem} If $f\in C^k[a,b]$ and $\a>0$, then
$${^CD_{a+}^{n-\a,\p}}\LD f(x)= {^CD_{a+}^{n,\p}} f(x)\quad \mbox{and} \quad {^CD_{b-}^{n-\a,\p}}\RD f(x)= {^CD_{b-}^{n,\p}}f(x).$$
\end{theorem}
\begin{proof} Attending that $\LD f(a)=0$ and applying Theorem \ref{relation}, we have
\begin{align*}
{^CD_{a+}^{n-\a,\p}}\LD f(x)& =  {D_{a+}^{n-\a,\p}}\LD f(x)\\
                            & =  \left(\frac{1}{\p'(x)}\frac{d}{dx}\right)^{n-[\a]}{I_{a+}^{\a-[\a],\p}} {I_{a+}^{[\a]+1-\a,\p}}\left(\frac{1}{\p'(x)}\frac{d}{dx}\right)^{[\a]+1} f(x)\\
                            & =  \left(\frac{1}{\p'(x)}\frac{d}{dx}\right)^{n-[\a]-1} \left(\frac{1}{\p'(x)}\frac{d}{dx}\right) {I_{a+}^{1,\p}} \left(\frac{1}{\p'(x)}\frac{d}{dx}\right)^{[\a]+1} f(x)\\
                            &=  {^CD_{a+}^{n,\p}}.
\end{align*}
\end{proof}

\section{Miscelious  results}
\label{sec:Miscelious}

In this section we study several properties of this fractional operator. To start, we obtain an integration by parts formula (Theorem \ref{Theorem 12}). In Theorem \ref{13}, we establish a relation between the fractional derivative of a function of order close to one, with its ordinary derivative of order one. In Theorems \ref{teo:fermat} and \ref{16}, we obtain the fractional versions of the Fermat's and the mean value theorems, respectively. We also refer to Theorem \ref{18}, where we prove a Taylor's formula for such operators.

\begin{theorem}\label{Theorem 12} Given $f \in C([a,b])$ and $g\in C^n([a,b])$, we have that for all  $\a>0$,
\begin{align*}\int_a^b f(x)  \LD g(x)\,dx &= \int_a^b \RRD  \left(\frac{f(x)}{\p'(x)}\right) g(x)\p'(x)\,dx\\
&\quad+\left[\sum_{k=0}^{n-1}\left(-\frac{1}{\p'(x)}\frac{d}{dx}\right)^k {I_{b-}^{n-\a,\p}} \left(\frac{f(x)}{\p'(x)}\right) g^{[n-k-1]}_\p(x)\right]_{x=a}^{x=b}\end{align*}
and
\begin{align*}\int_a^b f(x)  \RD g(x)\,dx &= \int_a^b \RLD \left(\frac{f(x)}{\p'(x)}\right) g(x)\p'(x)\,dx\\
&\quad +\left[\sum_{k=0}^{n-1}(-1)^{n-k}\left(\frac{1}{\p'(x)}\frac{d}{dx}\right)^k {I_{a+}^{n-\a,\p}} \left(\frac{f(x)}{\p'(x)}\right) g^{[n-k-1]}_\p(x)\right]_{x=a}^{x=b}.\end{align*}
\end{theorem}
\begin{proof} Using Dirichlet's formula, we obtain
\begin{align*}\int_a^b f(x)  \LD g(x)\,dx &= \frac{1}{\Gamma(n-\a)}\int_a^b \int_a^x f(x)(\p(x)-\p(t))^{n-\a-1}\frac{d}{dt}g^{[n-1]}_\p (t)\,dt \,dx\\
                         &= \frac{1}{\Gamma(n-\a)}\int_a^b \int_x^b f(t)(\p(t)-\p(x))^{n-\a-1}\,dt \cdot \frac{d}{dx}g^{[n-1]}_\p (x) \,dx\\
\end{align*}
Using integration by parts, we get
\begin{align*}
&\frac{1}{\Gamma(n-\a)}\left[\int_x^b f(t)(\p(t)-\p(x))^{n-\a-1}\,dt \cdot g^{[n-1]}_\p (x)\right]_{x=a}^{x=b}\\
&\quad -\frac{1}{\Gamma(n-\a)}\int_a^b \frac{d}{dx}\left(\int_x^b f(t)(\p(t)-\p(x))^{n-\a-1}\,dt\right)g^{[n-1]}_\p (x) \,dx\\
&=\frac{1}{\Gamma(n-\a)}\left[\int_x^b f(t)(\p(t)-\p(x))^{n-\a-1}\,dt \cdot g^{[n-1]}_\p (x)\right]_{x=a}^{x=b}\\
& \quad+\frac{1}{\Gamma(n-\a)}\int_a^b \left(-\frac{1}{\p'(x)}\frac{d}{dx}\right)\left(\int_x^b f(t)(\p(t)-\p(x))^{n-\a-1}\,dt\right)\cdot \frac{d}{dx}g^{[n-2]}_\p (x) \,dx.
\end{align*}
Using again integration by parts, the last formula is equal to
\begin{align*}
&\left[\sum_{k=0}^1\left(-\frac{1}{\p'(x)}\frac{d}{dx}\right)^k\frac{1}{\Gamma(n-\a)}\int_x^b f(t)(\p(t)-\p(x))^{n-\a-1}\,dt \cdot g^{[n-k-1]}_\p (x)\right]_{x=a}^{x=b}\\
& \quad +\frac{1}{\Gamma(n-\a)}\int_a^b \left(-\frac{1}{\p'(x)}\frac{d}{dx}\right)^2\left(\int_x^b f(t)(\p(t)-\p(x))^{n-\a-1}\,dt\right)\cdot \frac{d}{dx}g^{[n-3]}_\p (x) \,dx.
\end{align*}
Repeating the process, we get
\begin{align*}
&\left[\sum_{k=0}^{n-1}\left(-\frac{1}{\p'(x)}\frac{d}{dx}\right)^k\frac{1}{\Gamma(n-\a)}\int_x^b f(t)(\p(t)-\p(x))^{n-\a-1}\,dt \cdot g^{[n-k-1]}_\p (x)\right]_{x=a}^{x=b}\\
& \quad +\frac{1}{\Gamma(n-\a)}\int_a^b \left(-\frac{1}{\p'(x)}\frac{d}{dx}\right)^n\left(\int_x^b f(t)(\p(t)-\p(x))^{n-\a-1}\,dt\right)\cdot g(x)\p'(x) \,dx\\
& =\left[\sum_{k=0}^{n-1}\left(-\frac{1}{\p'(x)}\frac{d}{dx}\right)^k {I_{b-}^{n-\a,\p}} \left(\frac{f(x)}{\p'(x)}\right) g^{[n-k-1]}_\p(x)\right]_{x=a}^{x=b}\\
& \quad +\int_a^b \RRD  \left(\frac{f(x)}{\p'(x)}\right) g(x)\p'(x)\,dx.
\end{align*}
\end{proof}

For our next result, we recall the Weierstrass formula
$$\frac{1}{\Gamma(x)}=x \exp(\gamma x)\prod_{k=1}^\infty\left[\left(1+\frac{x}{k}\right)\exp\left(-\frac{x}{k}\right)\right],$$
where $x \in \mathbb R \setminus \{0,-1,-2,\ldots\}$, and
$$\gamma=\lim\left(1+\frac12+\ldots+\frac1n-\ln n\right)\approx 0.5772$$
is the Euler's constant. Also, since
$$\prod_{k=1}^\infty\left[\left(1+\frac{\epsilon}{k}\right)\exp\left(-\frac{\epsilon}{k}\right)\right]=1+O(\epsilon^2),$$
we get
$$\frac{1}{\Gamma(\epsilon+1)}=\frac{1}{\epsilon\Gamma(\epsilon)}=\exp(\gamma\epsilon)+O(\epsilon^2).$$

\begin{theorem}\label{13} Let $f\in C^1[a,b]$ be a function and $\epsilon\in(0,1)$ a real. Then,
$${^CD_{a+}^{1-\epsilon,\p}}f(x)=f^{[1]}_\p(x)$$
$$+\epsilon\left[\gamma f^{[1]}_\p(x) +f^{[1]}_\p(a)\ln(\p(x)-\p(a))
+\int_a^x \frac{d}{dt}\left(f^{[1]}_\p(t)\right)\ln(\p(x)-\p(t))\,dt\right]+O(\epsilon^2)$$
and
$${^CD_{b-}^{1-\epsilon,\p}}f(x)=-f^{[1]}_\p(x)$$
$$+\epsilon\left[-\gamma f^{[1]}_\p(x) -f^{[1]}_\p(b)\ln(\p(b)-\p(x))
+\int_x^b \frac{d}{dt}\left(f^{[1]}_\p(t)\right)\ln(\p(t)-\p(x))\,dt\right]+O(\epsilon^2).$$
\end{theorem}
\begin{proof} Integrating by parts and using the Taylor's formula for the exponential function, we get
\begin{align*}
{^CD_{a+}^{1-\epsilon,\p}}f(x)& = \frac{1}{\Gamma(\epsilon)}\int_a^x \frac{f'(t)}{\p'(t)}\p'(t)(\p(x)-\p(t))^{\epsilon-1}\,dt\\
           & =\frac{f'(a)}{\p'(a)\Gamma(\epsilon+1)}(\p(x)-\p(a))^{\epsilon}+ \frac{1}{\Gamma(\epsilon+1)}\int_a^x \frac{d}{dt} \left(\frac{f'(t)}{\p'(t)}\right)(\p(x)-\p(t))^{\epsilon}\,dt\\
           & =\frac{f'(a)}{\p'(a)\Gamma(\epsilon+1)}\exp(\epsilon\ln(\p(x)-\p(a)))\\
           & \quad + \frac{1}{\Gamma(\epsilon+1)}\int_a^x \frac{d}{dt} \left(\frac{f'(t)}{\p'(t)}\right)\exp(\epsilon\ln(\p(x)-\p(t)))\,dt\\
           & =\frac{f'(a)}{\p'(a)}\exp(\epsilon(\gamma+\ln(\p(x)-\p(a))))\\
           & \quad + \int_a^x \frac{d}{dt} \left(\frac{f'(t)}{\p'(t)}\right)\exp(\epsilon(\gamma+\ln(\p(x)-\p(t))))\,dt+O(\epsilon^2)\\
           & =\frac{f'(a)}{\p'(a)}(1+\epsilon(\gamma+\ln(\p(x)-\p(a))))\\
           & \quad + \int_a^x \frac{d}{dt} \left(\frac{f'(t)}{\p'(t)}\right)(1+\epsilon(\gamma+\ln(\p(x)-\p(t))))\,dt+O(\epsilon^2)\\
           &= \frac{f'(x)}{\p'(x)}+\epsilon\left[\gamma \frac{f'(x)}{\p'(x)}+\frac{f'(a)}{\p'(a)}\ln(\p(x)-\p(a))\right.\\
           & \quad  \left.
           +\int_a^x \frac{d}{dt}\left(\frac{f'(t)}{\p'(t)}\right)\ln(\p(x)-\p(t))\,dt\right]+O(\epsilon^2),
\end{align*}
obtaining the desired formula.
\end{proof}

For the Riemann--Liouville case, that is, when $\p(x)=x$, see \cite{Tarasov}. When $\a\approx 1^-$, this is known as low-level fractionality.

\begin{theorem} Given $\a\in(0,1)$ and $f\in C^1[a,b]$, we have
$$\LD (\p(x)f(x))=\p(x){^CD_{a+}^{\a,\p}}f(x)+{I_{a+}^{1-\a,\p}}f(x)-(1-\a){I_{a+}^{2-\a,\p}}{f^{[1]}_\p}(x)$$
and
$$\RD (\p(x)f(x))=\p(x){^CD_{b-}^{\a,\p}}f(x)-{I_{b-}^{1-\a,\p}}f(x)-(1-\a){I_{b-}^{2-\a,\p}}{f^{[1]}_\p}(x).$$
\end{theorem}
\begin{proof} We have
\begin{align*}
\LD (\p(x)f(x))&= \frac{1}{\Gamma(1-\a)}\int_a^x (\p(x)-\p(t))^{-\a}(\p(t)f(t))' (t)\,dt\\
&= \frac{1}{\Gamma(1-\a)}\int_a^x \p'(t)(\p(x)-\p(t))^{-\a}f(t)\,dt\\
 & \quad +\frac{1}{\Gamma(1-\a)}\int_a^x (\p(x)-\p(t))^{-\a}(\p(x)-(\p(x)-\p(t)))f'(t)\,dt\\
&= {I_{a+}^{1-\a,\p}}f(x)+\frac{\p(x)}{\Gamma(1-\a)}\int_a^x (\p(x)-\p(t))^{-\a}f'(t)\,dt\\
  & \quad -\frac{1}{\Gamma(1-\a)}\int_a^x (\p(x)-\p(t))^{1-\a}f'(t)\,dt\\
&=\p(x){^CD_{a+}^{\a,\p}}f(x)+{I_{a+}^{1-\a,\p}}f(x)-(1-\a){I_{a+}^{2-\a,\p}}{f^{[1]}_\p}(x).
\end{align*}
\end{proof}

For a similar problem dealing with the Riemann--Liouville fractional derivative, see \cite{Gazizov}.
The following result is a fractional version of Fermat's theorem. For the Riemann--Liouville case, we suggest  \cite{Atanackovic}.

\begin{theorem}\label{teo:fermat} Let $\a\in(0,1)$ be a real and $f\in C^1[a,b]$ a function. If $f(x^*)$ is maximum, then $\LD f(x^*) \geq0$ and $\RD f(x^*) \geq 0$.
\end{theorem}
\begin{proof} Since $f(x^*)$ is maximum, then $f$ is non-decreasing  in $[a,x^*]$ and it is non-increasing  in $[x^*,b]$. Let us prove the first inequality. Using the change of variables $\t=\p^{-1}(\p(x)-\p(t)+\p(a))$, we get
\begin{align*}
\LD f(x)&=\frac{1}{\Gamma(1-\a)}\int_a^x \frac{(\p(\t)-\p(a))^{-\a}f' (\p^{-1}(\p(x)-\p(\t)+\p(a)))}{\p'(\p^{-1}(\p(x)-\p(\t)+\p(a)))}\p'(\t)\,d\t\\
      &=\frac{-1}{\Gamma(1-\a)}\int_a^x\frac{d}{dt}\left(f(\p^{-1}(\p(x)-\p(\t)+\p(a)))\right)\\
      & \quad \cdot \left(\a\int_\t^x \p'(s)(\p(s)-\p(a))^{-\a-1}\,ds +\frac{1}{(\p(x)-\p(a))^\a}\right)d\t.
      \end{align*}
Integrating by parts, we obtain
\begin{align*}
& \frac{-1}{\Gamma(1-\a)}\left[\frac{f(a)}{(\p(x)-\p(a))^\a}-f(x)\left(\a\int_a^x \p'(s)(\p(s)-\p(a))^{-\a-1}\,ds +\frac{1}{(\p(x)-\p(a))^\a}\right)\right]\\
    & \quad + \frac{1}{\Gamma(1-\a)}\int_a^x f(\p^{-1}(\p(x)-\p(\t)+\p(a)))\left(-\frac{\a\p'(\t)}{(\p(\t)-\p(a))^{\a+1}}\right)d\t\\
       & = \frac{f(x)-f(a)}{\Gamma(1-\a)(\p(x)-\p(a))^\a}+\frac{\a}{\Gamma(1-\a)}
       \int_a^x \frac{f(x)-f(\p^{-1}(\p(x)-\p(\t)+\p(a)))}{(\p(\t)-\p(a))^{\a+1}}\p'(\t)\,d\t.
      \end{align*}
      Since $f(x^*) \geq f(a)$ and $f(x)\geq f(\p^{-1}(\p(x)-\p(\t)+\p(a)))$, for all $\t\in [a,x^*]$, it follows that $\LD f(x^*) \geq0$. The proof of the second inequality is similar, if one uses the change of variables $\t=\p^{-1}(\p(x)-\p(t)+\p(b))$.
\end{proof}
We remark that, in Theorem \ref{teo:fermat}, we can not conclude that the fractional derivative vanishes at the extremum point. Take, for example, $f(x)=2x-x^2$ and the kernel $\p(x)=x$, with $x\in[0,2]$. Then, $f$ satisfies the hypothesis of the Theorem with $x^*=1$, and e.g. for $\a=0.5$, we have
$${^CD_{0+}^{0.5,\p}} f(1)=\frac{2}{\Gamma(1.5)}-\frac{2}{\Gamma(2.5)}\approx 0.75>0.$$
Also, since $f(x)=2(2-x)-(2-x)^2$, then
$${^CD_{2-}^{0.5,\p}} f(1)=\frac{2}{\Gamma(1.5)}-\frac{2}{\Gamma(2.5)}\approx0.75>0.$$
It is clear, from the definition of the $\p$-Caputo fractional derivative, that if $f$ is non-decreasing (resp. non-increasing) in $[a,x^*]$, then $\LD f(x) \geq 0$ (resp $\leq0$) since $f'(x) \geq 0$ (resp $\leq 0$), for all $x\in [a,x^*]$.

The next result is a fractional mean value theorem for the $\p$-Caputo fractional derivative.

\begin{theorem}\label{16} Let $f\in C^1[a,b]$ and $\a\in(0,1)$. Then, for all $x \in (a,b)$, there exists some $c\in(a,x)$ and some $d\in(x,b)$ such that
\begin{align*}
f(x)&= f(a)+\LD f(c)\frac{(\p(x)-\p(a))^\a}{\Gamma(\a+1)}\\
&= f(b)+\RD f(d)\frac{(\p(b)-\p(x))^\a}{\Gamma(\a+1)}
\end{align*}
\end{theorem}
\begin{proof} By the mean value theorem for integrals, there exists some $c\in(a,x)$ such that
\begin{align*}
\LI \LD f(x)&= \frac{1}{\Gamma(\a)}\int_a^x \p'(t)(\p(x)-\p(t))^{\a-1}\LD f(t)\,dt \\
&=\LD f(c) \frac{1}{\Gamma(\a)}\int_a^x \p'(t)(\p(x)-\p(t))^{\a-1}\,dt \\
&=\LD f(c)\frac{(\p(x)-\p(a))^\a}{\Gamma(\a+1)}.
\end{align*}
On the other hand, by Theorem \ref{inverseoperation1}, we have
$$\LI \LD f(x)=f(x)-f(a),$$
and thus
$$f(x)-f(a)=\LD f(c)\frac{(\p(x)-\p(a))^\a}{\Gamma(\a+1)}.$$
\end{proof}

\begin{theorem}\label{aux3} Given $\a\in(0,1)$, $k\in\mathbb N$ and $f$ is such that all the fractional derivatives ${^CD_{a+}^{k\a,\p}}f$, ${^CD_{a+}^{(k+1)\a,\p}}f$, ${^CD_{b-}^{k\a,\p}}f$ and ${^CD_{b-}^{(k+1)\a,\p}}f$ exist and are continuous on $[a,b]$. Then, for all $x\in[a,b]$, we have
$${I_{a+}^{k\a,\p}} {^CD_{a+}^{k\a,\p}}f(x)- {I_{a+}^{(k+1)\a,\p}}{^CD_{a+}^{(k+1)\a,\p}}f(x)
=\frac{(\p(x)-\p(a))^{k\a}}{\Gamma(k\a+1)}{^CD_{a+}^{k\a,\p}}f(a),$$
and
$${I_{b-}^{k\a,\p}} {^CD_{b-}^{k\a,\p}}f(x)- {I_{b-}^{(k+1)\a,\p}}{^CD_{b-}^{(k+1)\a,\p}}f(x)
=\frac{(\p(b)-\p(x))^{k\a}}{\Gamma(k\a+1)}{^CD_{b-}^{k\a,\p}}f(b),$$
where ${^CD_{a+}^{k\a,\p}}=\LD \LD \ldots \LD$ ($k$ -times).
\end{theorem}
\begin{proof} Using the semigroup law for integrals and Theorem \ref{inverseoperation1}, we get
$${I_{a+}^{k\a,\p}} {^CD_{a+}^{k\a,\p}}f(x)- {I_{a+}^{(k+1)\a,\p}}{^CD_{a+}^{(k+1)\a,\p}}f(x)
={I_{a+}^{k\a,\p}}\left(  {^CD_{a+}^{k\a,\p}}f(x)- {I_{a+}^{\a,\p}}\LD{^CD_{a+}^{k\a,\p}}f(x)\right)$$
$$={I_{a+}^{k\a,\p}}\left(  {^CD_{a+}^{k\a,\p}}f(x)- {^CD_{a+}^{k\a,\p}}f(x)+{^CD_{a+}^{k\a,\p}}f(a)\right)= \frac{(\p(x)-\p(a))^{k\a}}{\Gamma(k\a+1)}{^CD_{a+}^{k\a,\p}}f(a).$$
\end{proof}

For a similar result involving the Caputo fractional derivative, that is, when $\p(x)=x$, we mention \cite{Odibat}. As a consequence of the previous result, we deduce a fractional Taylor's formula.

\begin{theorem}\label{18} Given $\a\in(0,1)$, $n\in\mathbb N$ and $f$ is such that ${^CD_{a+}^{k\a,\p}}f$ and ${^CD_{b-}^{k\a,\p}}f$ exist and are continuous, for all $k=0,1,\ldots,n+1$. Then, given $x\in[a,b]$, we have the following:
\begin{align*}
f(x)&= \sum_{k=0}^n \frac{(\p(x)-\p(a))^{k\a}}{\Gamma(k\a+1)}{^CD_{a+}^{k\a,\p}}f(a)
+\frac{{^CD_{a+}^{(n+1)\a,\p}}f(c)}{\Gamma((n+1)\a+1)}(\p(x)-\p(a))^{(n+1)\a}\\
&= \sum_{k=0}^n \frac{(\p(b)-\p(x))^{k\a}}{\Gamma(k\a+1)}{^CD_{b-}^{k\a,\p}}f(b)
+\frac{{^CD_{b-}^{(n+1)\a,\p}}f(d)}{\Gamma((n+1)\a+1)}(\p(b)-\p(x))^{(n+1)\a},\\
\end{align*}
for some $c\in(a,x)$ and some $d\in(x,b)$.
\end{theorem}
\begin{proof} By Theorem \ref{aux3}, we have
$$\sum_{k=0}^n\left({I_{a+}^{k\a,\p}} {^CD_{a+}^{k\a,\p}}f(x)- {I_{a+}^{(k+1)\a,\p}}{^CD_{a+}^{(k+1)\a,\p}}f(x)\right)
=\sum_{k=0}^n\frac{(\p(x)-\p(a))^{k\a}}{\Gamma(k\a+1)}{^CD_{a+}^{k\a,\p}}f(a).$$
So, we conclude that
$$f(x)=\sum_{k=0}^n\frac{(\p(x)-\p(a))^{k\a}}{\Gamma(k\a+1)}{^CD_{a+}^{k\a,\p}}f(a)+{I_{a+}^{(n+1)\a,\p}}{^CD_{a+}^{(n+1)\a,\p}}f(x).$$
Using the mean value theorem for integrals, we have
\begin{align*}
{I_{a+}^{(n+1)\a,\p}}{^CD_{a+}^{(n+1)\a,\p}}f(x) & = \frac{1}{\Gamma((n+1)\a)}\int_a^x \p'(t)(\p(x)-\p(t))^{(n+1)\a-1}{^CD_{a+}^{(n+1)\a,\p}}f(t)\,dt\\
& = \frac{{^CD_{a+}^{(n+1)\a,\p}}f(c)}{\Gamma((n+1)\a+1)}(\p(x)-\p(a))^{(n+1)\a},
\end{align*}
for some $c\in(a,x)$, ending the proof.
\end{proof}

\section{A numerical tool}
\label{sec:numerical}

One issue that usually arises when dealing with operators of fractional type is that it is extremely difficult to analytically solve the problems.
Although we already have several results on the existence and uniqueness theorems for fractional differential equations, necessary conditions for solutions of optimal control problems, when we intend to determine the solution for the problem, numerical methods are  usually required to find an approximation for it. In this section, we present a method, easy to use, that allows us to rewrite the fractional problem as an ordinary one, without the dependence on fractional operators. The idea is to approximate the fractional derivative by a sum, which depends only on the derivatives of integer order. This was already done for the Riemann--Liouville and Caputo fractional operators \cite{Almeida1,Atanackovic1}, and for the Hadamard fractional operators \cite{Almeida3}. We also refer to \cite{Tavares}, where problems dealing with the Caputo fractional derivative of variable order are solved using such expansions. In this paper, we generalize some of the previous works by considering a Caputo type fractional derivative with respect to another function of constant fractional order.

\begin{theorem}\label{teo:approx} Let $f\in C^{n+1}[a,b]$ be a function, $\a>0$ arbitrary, and $N$ a positive integer. Then,
$$\LD f(x)=A_N(\p(x)-\p(a))^{n-\a}\f (x)-\sum_{k=1}^NB_k (\p(x)-\p(a))^{n-\a-k}V_k(t)+\overline{E}_N(t)$$
and
$$\RD f(x)=(-1)^n A_N(\p(b)-\p(x))^{n-\a}\f (x)-\sum_{k=1}^NB_k (\p(b)-\p(x))^{n-\a-k}W_k(t)+\overline{\overline{E}}_N(t),$$
where
\begin{align*}
A_N:=&\frac{1}{\Gamma(n-\a+1)} \left[1+\sum_{k=1}^N(-1)^k \binom{n-\a}{k}\right],\\
B_k:=&\frac{(-1)^k}{\Gamma(n-\a+1)} \binom{n-\a}{k},\\
V_k(t):=& \int_a^x k \p'(t)(\p(t)-\p(a))^{k-1}\f(t)\,dt,\\
W_k(t):=& \int_x^b k (-1)^n \p'(t)(\p(b)-\p(t))^{k-1}\f(t)\,dt,\\
\end{align*}
and
$$\lim_{N\to\infty }\overline{E}_N(t)=\lim_{N\to\infty }\overline{\overline{E}}_N(t)=0, \quad \forall t\in[a,b].$$
\end{theorem}
\begin{proof} Starting with the definition of $\p$-Caputo fractional derivative, and using Theorem \ref{alternativeFormula}, we get
\begin{align*}
\LD f(x)& =\frac{(\p(x)-\p(a))^{n-\a}}{\Gamma(n-\a+1)}\f(a)+\frac{1}{\Gamma(n-\a+1)}\int_a^x(\p(x)-\p(t))^{n-\a}\frac{d}{dt}\f(t)\,dt\\
& =\frac{(\p(x)-\p(a))^{n-\a}}{\Gamma(n-\a+1)}\f(a)+\frac{(\p(x)-\p(a))^{n-\a}}{\Gamma(n-\a+1)}
                 \int_a^x\left(1-\frac{\p(t)-\p(a)}{\p(x)-\p(a)}\right)^{n-\a}\frac{d}{dt}\f(t)\,dt.\\
\end{align*}
Applying now Newton's generalised binomial theorem, we can expand the binomial inside the integral, obtaining
\begin{align*}
\LD f(x)&=\frac{(\p(x)-\p(a))^{n-\a}}{\Gamma(n-\a+1)}\f(a)+\frac{(\p(x)-\p(a))^{n-\a}}{\Gamma(n-\a+1)}\\
    & \quad \times \sum_{k=0}^N\frac{(-1)^k}{(\p(x)-\p(a))^k}\binom{n-\a}{k}
        \int_a^x(\p(t)-\p(a))^k \frac{d}{dt}\f(t)\,dt+\overline{E}_N(t),\\
\end{align*}
where
$$\overline{E}_N(t):= \frac{(\p(x)-\p(a))^{n-\a}}{\Gamma(n-\a+1)}\sum_{k=N+1}^\infty\frac{(-1)^k}{(\p(x)-\p(a))^k}\binom{n-\a}{k}
        \int_a^x(\p(t)-\p(a))^k \frac{d}{dt}\f(t)\,dt.$$
        Splitting the sum into the first term $k=0$ and the remaining terms $k=1,\ldots,N$, and then using integration by parts, we get
\begin{align*}
\LD f(x)&=\frac{(\p(x)-\p(a))^{n-\a}}{\Gamma(n-\a+1)}\f(a)+\frac{(\p(x)-\p(a))^{n-\a}}{\Gamma(n-\a+1)}(\f(x)-\f(a))\\
   & \quad +\frac{(\p(x)-\p(a))^{n-\a}}{\Gamma(n-\a+1)}\sum_{k=1}^N\frac{(-1)^k}{(\p(x)-\p(a))^k}\binom{n-\a}{k}\\
  & \quad \times \left[(\p(x)-\p(a))^k \f(x)-  \int_a^x k \p'(t)(\p(t)-\p(a))^{k-1}\f(t)\,dt\right]+\overline{E}_N(t),\\
      & =A_N(\p(x)-\p(a))^{n-\a}\f (x)-\sum_{k=1}^NB_k (\p(x)-\p(a))^{n-\a-k}V_k(t)+\overline{E}_N(t)
\end{align*}
We now seek an upper bound for the error formula $\overline{E}_N(t)$. Let
$$M:= \max_{t\in[a,b]}\left| \frac{d}{dt}\f(t) \right|.$$
Once
$$\sum_{k=N+1}^\infty\left|\binom{n-\a}{k}\right| \leq \sum_{k=N+1}^\infty \frac{\exp((n-\a)^2+n-\a)}{k^{n-\a+1}}$$
$$\leq \int_N^\infty  \frac{\exp((n-\a)^2+n-\a)}{k^{n-\a+1}} dk= \frac{\exp((n-\a)^2+n-\a)}{(n-\a)k^{n-\a}}$$
and
$$\frac{\p(t)-\p(a)}{\p(x)-\p(a)}\leq 1, \forall t \in [a,x],$$
since $\p$ is an increasing function, it  follows that
$$|\overline{E}_N(t)|\leq M \frac{(\p(x)-\p(a))^{n-\a}(x-a) \exp((n-\a)^2+n-\a)}{(n-\a)\Gamma(n-\a+1)N^{n-\a}},$$
and so
$$\lim_{N\to\infty }\overline{E}_N(t)=0.$$
\end{proof}

Theorem \ref{teo:approx} provides a numerical method to deal with fractional type problems. It consists in approximating the $\p$-Caputo fractional derivatives of $f$, $\LD f(x)$ and $\RD f(x)$, by a sum of functions that involves the first-order derivative of $f$ only:
$$\LD f(x)\approx A_N(\p(x)-\p(a))^{n-\a}\f (x)-\sum_{k=1}^NB_k (\p(x)-\p(a))^{n-\a-k}V_k(t),$$
and
$$\RD f(x)\approx (-1)^n A_N(\p(b)-\p(x))^{n-\a}\f (x)-\sum_{k=1}^NB_k (\p(b)-\p(x))^{n-\a-k}W_k(t),$$
and the error decreases as $N\to\infty$.
Using this, we can rewrite the fractional problem as an ordinary one, and then apply any analytical or numerical method available to solve the problem.
For example, consider the kernel of Caputo--Hadamard type $\p(x)=\ln(x+1)$, with $x\in[0,5]$. If we consider $\a\in(0,1)$, then given a function of class $C^2$, we have
\begin{equation}\label{EqApprox}{^CD_{0+}^{\a,\p}}f(x)\approx A_N(x+1)(\ln(x+1))^{1-\a}f'(x)-\sum_{k=1}^NB_k (\ln(x+1))^{1-\a-k}V_k(t),\end{equation}
where
\begin{align*}
A_N:=&\frac{1}{\Gamma(2-\a)} \left[1+\sum_{k=1}^N(-1)^k \binom{1-\a}{k}\right],\\
B_k:=&\frac{(-1)^k}{\Gamma(1-\a)} \binom{1-\a}{k},\\
V_k(t):=& \int_0^x k (\ln(t+1))^{k-1}f'(t)\,dt.
\end{align*}
In Figure \ref{figApproxim}, we compare the exact expression of the fractional derivative of $f(x)=\ln^2(x+1)$, for $x\in[0,5]$, with several numerical approximations of it, as given by Eq. \eqref{EqApprox}, for different values of $N$, with fractional order $\a=0.5$.

\begin{figure}[!htbp]
\centering
\includegraphics[width=10cm]{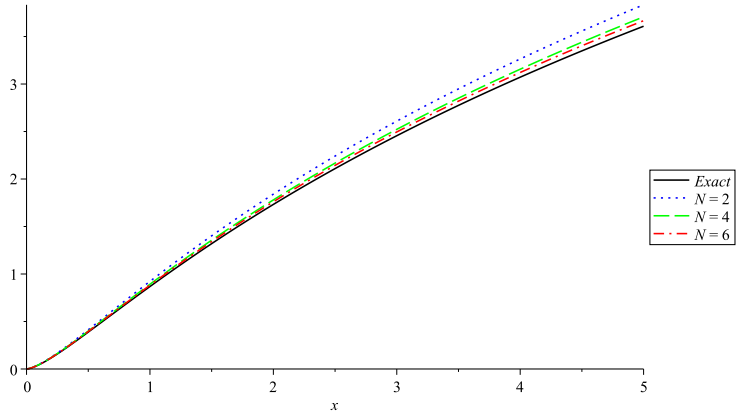}
\caption{Exact vs. numerical approximations of ${^CD_{0+}^{\a,\p}}f$, with $f(x)=\ln^2(x + 1)$.}
\label{figApproxim}
\end{figure}

We now exemplify how the method can be useful to solve fractional problems. Consider a nonlinear fractional differential equation of order $\a>0$:
$$\LD f(x)=g(x,f(x)),$$
on a finite interval $[a,b]$, with the initial conditions
$$f^{(k)}(a)=f_k, \quad k=0,\ldots,n-1, \quad \mbox{where } n=[\a]+1.$$
Using the approximation given by Theorem \ref{teo:approx}, we replace $\LD f(x)$ by such formula, and with this we get a system of ODE's:
$$\left\{\begin{array}{l}
A_N(\p(x)-\p(a))^{n-\a}\f (x)-\sum_{k=1}^NB_k (\p(x)-\p(a))^{n-\a-k}V_k(t)= g(x,f(x))\\
V_k'(t)=  k \p'(x)(\p(x)-\p(a))^{k-1}\f(x), \quad k=1,\ldots,N,\\
\end{array}\right.$$
with $n+N$  initial conditions
$$f^{(k)}(a)=f_k, \quad k=0,\ldots,n-1, \quad \mbox{and} \quad V_k(a)=0, \quad k=1,\ldots,N.$$

For example, fix $\a=0.5$ and $\p(x)=\ln(x+1)$, with $x\in[0,5]$, and consider the Cauchy problem

\begin{equation}\label{FDEproblem}\left\{\begin{array}{l}
{^CD_{0+}^{0.5,\p}}f(x)+f(x)=\frac{2}{\Gamma(2.5)}\ln^{1.5}(x+1) +\ln^2(x+1), \quad x\in[0,5]\\
f(0)=0.
\end{array}\right.\end{equation}
The solution of this problem is the function $f(x)=\ln^2(x+1)$. Applying the decomposition formula \eqref{EqApprox}, and fixing an integer $N$, we rewrite this problem as an ordinary system of ODE's:
\begin{equation}\label{cauchy}\left\{\begin{array}{l}
A_N(x+1)\ln^{0.5}(x+1)f'(x)-\sum_{k=1}^NB_k \ln^{0.5-k}(x+1)V_k(t)+f(x)\\
 \quad = \frac{2}{\Gamma(2.5)}\ln^{1.5}(x+1) +\ln^2(x+1)\\
V_k'(t)=k \ln^{k-1}(x+1)f'(x), \quad k=1,\ldots,N\\
f(0)=0\\
V_k(0)=0, \quad k=1,\ldots,N,
\end{array}\right.\end{equation}
where
\begin{align*}
A_N:=&\frac{1}{\Gamma(1.5)} \left[1+\sum_{k=1}^N(-1)^k \binom{0.5}{k}\right],\\
B_k:=&\frac{(-1)^k}{\Gamma(0.5)} \binom{0.5}{k}.
\end{align*}

The results are shown in Figure \ref{figFDE}, for $N=2,4,6$. As we can see, as $N$ increases, the solution of the Cauchy problem \eqref{cauchy} is approaching the solution of the  FDE \eqref{FDEproblem}.

\begin{figure}[!htbp]
\centering
\includegraphics[width=10cm]{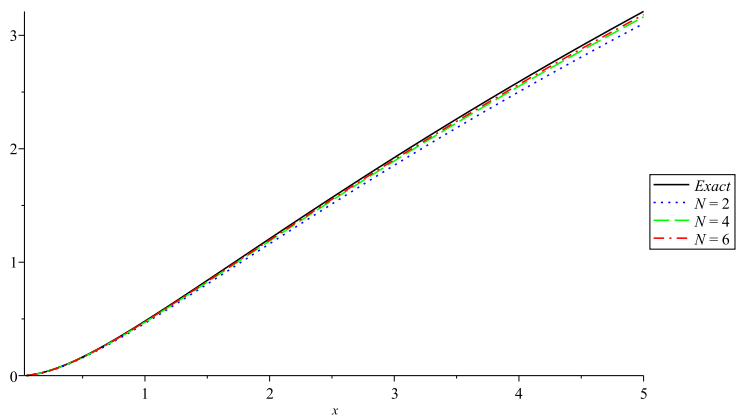}
\caption{Exact vs. numerical approximations for the FDE.}
\label{figFDE}
\end{figure}

\section{An application}
\label{sec:popul}

In this section we analise a model described by a fractional differential equation, and we show how fractional derivatives with respect to another function may be useful. The simplest model, but important historically, to describe the population growth  is the  Malthusian law, proposed in 1798 by the English economist Thomas Malthus \cite{Malthus}. It is given by the differential equation $N'(t)=\lambda N(t)$, where $\lambda$ is the population growth rate (equal to the difference between the birth and mortality rates), sometimes called Malthusian parameter, and $N(t)$ is the number of individuals in a population at time $t$. It is assumed that $\lambda$ is constant, and so if $N_0$ denotes the initial population size, then the solution to this Cauchy problem is the exponential function
\begin{equation}\label{popul-classical}N(t)=N_0\exp(\lambda t).\end{equation}
As is well know, fractional differential equations are sometimes more suitable to model natural situations, namely when there are constraints on the dynamics, when the conditions are not ideal, etc. For this, consider the FDE obtained from the Malthusian law of population growth, by replacing the first-order derivative by the $\p$-Caputo fractional derivative with respect to $\p$: ${^CD_{0+}^{\a,\p}} N(t)=\lambda N(t)$. From Lemma \ref{ex:E}, the solution of this FDE, together with the initial condition $N(0)=N_0$,  is the function
\begin{equation}\label{popul-fractional}N(t)=N_0E_\a(\lambda(\p(t)-\p(0))^\a).\end{equation}
If $\p(x)=x$, that is, if we are in the presence of the usual Caputo fractional derivative, then $N(t)=N_0E_\a(\lambda t^\a)$. This case was considered in \cite{Almeida2}, and it was proven that the FDE was more efficient in modelling the population growth than the ODE. In this work we go further, and exemplify that, when considering function $N$ as given by \eqref{popul-fractional}, we have a better accuracy on the model. The purpose is to determine the  best-fitting curve $N$ (that is, find the parameters $\lambda$ and $\a$) by minimizing the sum of the squares of the offsets of the points from the curve. We use the Maple \textit{NonlinearFit} command, that fits a nonlinear model in the model parameters of the data by minimizing the least-squares error: if $x_i$ are the observed values and $y_i$ the corresponding expected values, then we intend to minimize the sum $E:=\sum(x_i-y_i)^2$. The initial approximations for the routine were found, based on the data of the model.

The world population from 1910 until nowadays is given in Table \ref{tab1} \cite{UN}:
\begin{center}\begin{tabular}{|c|c|c|c|c|c|c|c|c|c|c|c|}
\hline
Year & 1910 & 1920 & 1930 & 1940 & 1950 & 1960 & 1970 & 1980 & 1990 & 2000 & 2010\\
\hline
Population &  1750 & 1860& 2070& 2300& 2520& 3020& 3700& 4440& 5270& 6060& 6790\\
\hline
\end{tabular}\captionof{table}{World population  (in millions) from year 1910 until 2010.}\label{tab1}
\end{center}

For the classical model \eqref{popul-classical}, with $x=0$ corresponding to the year 1910,  the best value is
$$\lambda \approx 0.13425 \quad \mbox{and the error is } \,E\approx6.75875\times 10^5.$$
Consider now the fractional model \eqref{popul-fractional}. We seek the values of the parameters $\lambda$ and $\a$ for which the error obtained is minimum. If we start by choosing the kernel $\p(x)=x$, that is, when we consider the Caputo fractional derivative ${^CD_{a+}^{\a}}N(t)$, the solution is given by
$$\lambda\approx0.085100, \quad \a\approx1.38935, \quad \mbox{with error }  E\approx 1.90896\times10^5.$$
Using this model, we see that the fractional model fits better with the data than the ordinary one.
Now, let us test other kernels and see if we can improve the efficiency of the model.
First, we consider $\p(x)=(x+1)^b$, with $b>0$. In Table \ref{tab2} we present some of the results:
\begin{center}\begin{tabular}{|c|c|c|c|}
\hline
$ b$ & $\lambda$ & $\a$ & Error\\
\hline
1.1 & 0.072991 &1.24897 & $ 2.13476\times10^5$\\
\hline
0.9 & 0.10613 & 1.55241 & $1.69172 \times 10^5$\\
\hline
0.8 & 0.14517 & 1.74137 & $1.48784\times10^5$\\
\hline
\end{tabular}\captionof{table}{Fractional models with kernel $\p(x) = (x + 1)^b$, for different values of $b$.}\label{tab2}
\end{center}

We can go further and determine the best value of $b$ for which the solution \eqref{popul-fractional}, involving the kernel $\p(x)=(x+1)^b$, with $b>0$, is closer to the data. In this case, we have the values
$$\lambda\approx0.26821, \quad \a\approx2.05784, \quad b\approx0.66734, \quad \mbox{with error } E\approx 1.26039\times10^5.$$
Other kinds of kernels could be considered. For example, for $\p(x)=\ln(x+1)$, we have the optimal values
$$\lambda\approx2.79881, \quad \a\approx4.44388 \quad \mbox{with error } E\approx 8.2257\times 10^4.$$ For $\p(x)=\sin(x/10)$, we have
$$\lambda\approx 5.35404, \quad \a\approx 1.93015 \quad \mbox{with error } E\approx 5.3735\times 10^4.$$

In Figure \ref{FigPopulacao}, we present the graphs of some of the fractional models, and the data given in Table \ref{tab1}.

\begin{figure}[!htbp]
\centering
\includegraphics[width=13.6cm]{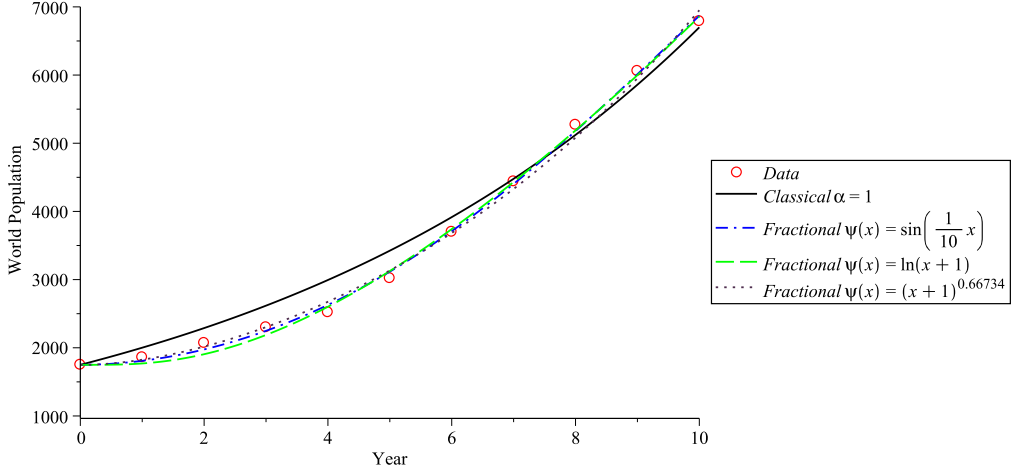}
\caption{World Population Growth model.}
\label{FigPopulacao}
\end{figure}

We can use our models to see how close they are to the projected population in 2015. In that year, the world population was estimated in 7350 millions. We compute  the error between the estimated value and the projected value by the following formula
$$E=\frac{|P(10.5)-7350|}{7350}\times100,$$
where $P$ is the one of the models that we obtained in this section. In Table \ref{tab3} we present the projected values, as well the error for each one.
\begin{center}\begin{tabular}{|c|c|c|}
\hline
Model & Projected & Error\\
\hline
Classical $\a=1$ & 7165 & 2.51382\\
\hline
Fractional $\p(x)=\sin(x/10)$ & 7294 & 0.75694\\
\hline
Fractional $\p(x)=\ln(x+1) $& 7302 & 0.65251\\
\hline
Fractional $\p(x)=(x+1)^{0.66734}$ & 7503 & 2.07646\\
\hline
\end{tabular}\captionof{table}{Projected population in 2015.}\label{tab3}
\end{center}

Now, we test our fractional models to a shorter period of time. If we consider the world population from the year 2000 until 2010, given by Table \ref{tab20},

\begin{center}\begin{tabular}{|c|c|c|c|c|c|c|c|c|c|c|c|}
\hline
Year & 2000& 2001 & 2002 & 2003 & 2004 & 2005 & 2006 & 2007 & 2008 & 2009 & 2010 \\
\hline
Population & 6060	&6165	&6242&	6319&	6396&	6473	&6551&	6630&	6709&	6788&	6790\\
\hline
\end{tabular}\captionof{table}{World population  (in millions) from year 2000 until 2010.}\label{tab20}
\end{center}
we get that for the classical model, the error is
$$E\approx 1.02223\times 10^4.$$
In Table \ref{tab21} we show the results when we model this new problem by a FDE, for the different kernels.
\begin{center}\begin{tabular}{|c|c|}
\hline
Model  & Error\\
\hline
Fractional $\p(x)=x$  & $2.98208\times 10^3$\\
\hline
Fractional $\p(x)=\sin(x/10)$  & $2.01593\times 10^3$\\
\hline
Fractional $\p(x)=\ln(x+1) $ & $3.70666\times 10^3$\\
\hline
Fractional $\p(x)=(x+1)^{0.56949}$  & $2.68650\times 10^3$\\
\hline
\end{tabular}\captionof{table}{Errors with respect to the FDE, from  year 2000 until 2010.}\label{tab21}
\end{center}
Thus, even when we consider short periods of time, the FDE is more accurate than the classical model.

\section*{Conclusion and future work}

The aim of this manuscript was to suggest a new fractional derivative with respect to another function, in the sense of Caputo derivative.
We derived some important properties of this new operator, and we established conditions for which the semigroup laws are valid. Numerical approximations have been provided, that allows us to solve any kind of problem, by considering a new one depending on ordinary derivatives only.
The new derivative was used to model the world population growth, and we saw that the choice of the fractional derivative is important for the efficiency of the method.

One interesting question is, given some experimental data, which is the kernel of the derivative that better describes the dynamics of the problem? One other possible generalization is to consider the fractional order as a function of time $\alpha(t)$, and determine the order that fits closer to the model. These and other questions will be treated in the future.


\section*{Acknowledgments}

The author is very grateful to Lu\'{i}s Machado, for a careful and thoughtful reading of the manuscript, and to three anonymous referees, for valuable remarks and comments that improved this paper.
Work supported by Portuguese funds through the CIDMA - Center for Research and Development in Mathematics and Applications, and the Portuguese Foundation for Science and Technology (FCT-Funda\c{c}\~ao para a Ci\^encia e a Tecnologia), within project UID/MAT/04106/2013.


\end{document}